\documentclass[12pt]{article}
\usepackage[utf8]{inputenc}
\usepackage[letterpaper,top=2cm,bottom=2cm,left=3cm,right=3cm,marginparwidth=1.75cm]{geometry}
\usepackage{graphicx}
\usepackage{graphics}
\usepackage{amsfonts}
\usepackage{amscd}
\usepackage{amsthm}
\usepackage{indentfirst}
\usepackage[all]{xy}
\usepackage{amssymb, amsmath, amsthm, amsgen, amstext, amsbsy, amsopn,dsfont}
\usepackage{epic,eepic}
\usepackage{enumitem} 
\usepackage{color}
\usepackage[dvipsnames]{xcolor}
\usepackage{ dsfont }
\usepackage[colorlinks=true, allcolors=blue]{hyperref}
\usepackage{comment}
\usepackage{soul}

\theoremstyle{plain}
\newtheorem{thm}{Theorem}

\newtheorem{cor}{Corollary}
\newtheorem{lemma}{Lemma}

\theoremstyle{definition}
\newtheorem{definition}[thm]{Definition}
\newtheorem{rem}{Remark}
\newtheorem{setup}[thm]{Setup}

\newcommand{\R}{{\mathbb{R}}}
\newcommand{\CP}{{\mathbb{CP}}}
\newcommand{\RP}{{\mathbb{RP}}}

\newcommand{\T}{{\mathbb{T}}}
\newcommand{\Z}{{\mathbb{Z}}}
\newcommand{\N}{{\mathbb{N}}}
\newcommand{\C}{{\mathbb{C}}}

\newcommand{\Per}{\operatorname{Per}}
\newcommand{\Ham}{\operatorname{Ham}}
\newcommand{\supp}{\operatorname{supp}}
\newcommand{\tHam}{\widetilde{\operatorname{Ham}}}

\newcommand{\cA}{\mathcal{A}}
\newcommand{\Diff}{\operatorname{Diff}}
\newcommand{\Spec}{\operatorname{Spec}}
\newcommand{\Fix}{\operatorname{Fix}}
\newcommand{\Symp}{\operatorname{Symp}}
\newcommand{\id}{\mathds{1}}

\newcommand{\tV}{\tilde{\varphi}}

\title{Ergodic closing lemmas and invariant Lagrangians}
\author{Erman \c C\. inel\. i, Sobhan Seyfaddini, and Shira Tanny}

\begin{document}

\maketitle
\begin{abstract}
    Motivated by the ergodic closing lemma of Mañé, we investigate the  $C^\infty$  closing lemma in higher-dimensional Hamiltonian systems, with a focus on the statistical behavior of periodic orbits generated by  $C^\infty$-small perturbations. We demonstrate that, under certain Floer-theoretic conditions, invariant or recurrent Lagrangian submanifolds can give rise to periodic orbits whose statistical properties are controllable. For instance, we show that for Hamiltonian systems preserving the zero section in $T^*\mathbb{T}^n$, $C^\infty$ generically, there exist periodic orbits converging to  an invariant measure supported on the zero section. 
\end{abstract}
\tableofcontents

\section{Introduction}
\label{sec:intro}
This paper is motivated by the study of the smooth closing lemma on higher dimensional symplectic manifolds. 

In the conservative context, where a symplectic or volume form is preserved, the $C^k$ closing lemma is typically stated as follows: for any open set $U$, there exists a $C^k$-small perturbation of the system that produces a periodic point or orbit passing through $U$. While the $C^1$ closing lemma was established some time ago \cite {Pugh1, Pugh2, Pugh-Robinson}, the higher regularity cases remain unresolved, except in the case of area-preserving maps on surfaces, where the $C^\infty$ version was proven using methods from low-dimensional symplectic topology \cite{i2015,asaoka2016c,cristofaro2021smooth,edtmair2021pfh}.

In higher dimensions, the $C^\infty$ closing lemma remains largely unresolved. Aside from a few specific cases (see \cite{cineli2022strong, CDPT,   CT2023, Xue}), there is no general method for detecting the creation of periodic orbits through $C^\infty$-small perturbations. In this paper, we demonstrate that, under certain Floer theoretic conditions, $C^\infty$-small perturbations of Hamiltonian systems with invariant or recurrent Lagrangian submanifolds  generate periodic orbits in arbitrarily small neighborhoods of the Lagrangian. Moreover, we go further by analyzing the statistical properties of these newly created periodic orbits—this is a central contribution of our work.

\subsection*{Statistical behavior of periodic orbits.}  In $C^1$ dynamics, there exist techniques for controlling the statistical behavior of periodic orbits that arise from carefully constructed perturbations.  Using such techniques one can show that, $C^1$ generically, periodic orbits are dense among ergodic measures; this is a consequence of the ergodic closing lemma of Ma\~n\'e \cite{mane1982ergodic} (see \cite[Sec. 4.1]{ABC}). 

As far as we know, there are no results pertaining to the statistics of periodic orbits in the context of the $C^\infty$ closing lemma on higher dimensional symplectic or contact manifolds.\footnote{In contrast to this, in the case of surfaces and 3D Reeb flows there exist results of this nature, due to Irie \& Prasad, proving $C^\infty$ generic equidistribution of periodic orbits; see \cite{Irie-equidist, prasad2022}.}   One of our main goals in this paper is to show that, in the presence of invariant Lagrangians, it is indeed possible to exert some control over the statistical behavior of periodic orbits.  Here is a sample result.

Consider $\Ham_c (T^*\T^n, \omega_{0})$, the set of compactly supported Hamiltonian diffeomorphisms of the cotangent bundle of the $n$-torus, equipped with its canonical symplectic structure.  Let $\Ham_{T_0}(T^*\T^n, \omega_0) := \{ \psi \in \Ham_c (T^*\T^n, \omega_0): \psi(T_0) =T_0 \}$, where $T_0$ denotes the zero section.

\begin{cor}\label{cor:cotangent}
    A $C^\infty$-generic element $\psi  \in \Ham_{T_0} (T^*\T^n, \omega_0)$ possesses a sequence of periodic orbits converging to an invariant measure supported on $T_0$. 
\end{cor}
As explained below, the above is an immediate consequence of Theorem \ref{thm:closing_near_heavy}.

\subsection*{Connections to KAM theory.}
Hamiltonian systems with invariant Lagrangian tori are studied extensively in the context of KAM theory. Let $H: T^*\T^n \rightarrow \R$ be a (sufficiently non-degenerate)  integrable Hamiltonian.  The phase space $T^*\T^n$ is foliated by Lagrangian tori which are invariant under the flow of $H$. According to  KAM theory, a significant  proportion of these invariant Lagrangian tori  persist for any $\tilde H: T^*\T^n \rightarrow \R$ which is sufficiently $C^\infty$-close to $H$. On these surving tori, which are known as KAM tori, 
the flow of $\tilde H$ is conjugate to an (Diophantine) irrational rotation and hence, in particular, is minimal. 

According to Conley and Zehnder (see the proof of \cite[Thm.\ 1.1]{Conley-Zehnder}), these  KAM tori are contained in the closure of the periodic orbits of $\tilde H$.  Hence, in the integrable setting Corollary \ref{cor:cotangent} can be deduced from KAM techniques. The results of this article, in particular Corollary \ref{cor:cotangent} \& Theorem \ref{thm:closing_near_heavy}, tell us that a similar conclusion holds in the non-integrable setting, where KAM theory is no longer applicable. Namely, under certain symplectic topological assumptions on the invariant Lagrangian and the ambient symplectic manifold, but with no integrability assumption on the map, we show that an invariant Lagrangian, which happens to be minimal, is contained in the closure of the periodic points of an arbitrarily small $C^\infty$-perturbation; see item (iv) in Remark \ref{rem:first remark}.

\subsection{Invariant Lagrangians}
 Let $(M, \omega)$ be a closed and connected symplectic manifold.  We suppose throughout the article that $(M, \omega)$ is {\bf rational}, meaning $\omega(\pi_2(M)) \subset \R$ is  discrete. The group of Hamiltonian diffeomorphisms $\Ham(M, \omega)$ admits, possibly several,  conjugation invariant pseudo-norms, commonly denoted by $\gamma$ and known as spectral pseudo-norms.  These were introduced in the works of  Viterbo, Schwarz and Oh \cite{viterbo92, schwarz, oh}. We say that $(M, \omega)$ is {\bf bounded} if it admits a bounded spectral pseudo-norm; see Definition \ref{defn:bdd} for further details.  Examples of bounded manifolds include the projective space $\C P^n$, and more generally all toric symplectic manifolds, as well as blow ups of closed symplectic manifolds \cite{usher2011deformed}.

Our results concern Hamiltonian systems admitting invariant subsets which are {\bf heavy} in the sense of Entov-Polterovich \cite{entov2009rigid}; this is a Floer theoretic notion, which we define precisely in Section \ref{sec:heaviness}, and it is typically satisfied by  Lagrangian submanifolds with non-vanishing Floer homology. The Clifford torus in $\C P^n$ is an example.

Examples of pairs $(M,L)$ of a rational symplectic manifold and a Lagrangian submanifold, such that  $M$ is bounded and $L$ is heavy includes $\C P^n$ with $L$ the Clifford torus and more generally certain toric manifolds with $L$ being a certain fiber of the moment map. There also exist examples where the heavy set $L$ is non-Lagrangian; one such example is the co-dimension one skeleton of a sufficiently fine triangulation.  See Section \ref{subsec:FH} for more details.

\medskip

Given $L \subset M$, let  $\Ham_L(M,\omega) := \{\psi \in \Ham(M, \omega): \psi(L) =L \}.$
Given $\psi \in \Ham_L(M,\omega)$, let $\Ham_{L, \psi}(M,\omega):= \{\phi \in \Ham_L(M,\omega): \phi|_L = \psi|_L \}.$

\begin{thm}\label{thm:closing_near_heavy}
Let  $(M, \omega)$ be rational and bounded.  Suppose that $L \subset M$ is heavy.

\begin{enumerate}
    \item A $C^\infty$-generic element $\psi \in \Ham_L(M,\omega)$,  possesses a sequence of periodic points converging to an invariant measure supported on $L$.

    \item Fix $\psi \in \Ham_L(M,\omega)$.  A $C^\infty$-generic element $\psi' \in \Ham_{L, \psi}(M,\omega)$,  possesses a sequence of periodic points converging to an invariant measure supported on $L$.
\end{enumerate}
\end{thm}

Convergence of a sequence of periodic orbits to a measure is recalled in Section \ref{sec:measures}.  As we explain in Section \ref{sec:Herman-heaviness}, the boundedness assumption appears to be essential.  The rationality assumption is likely removable; we comment on this in Remark \ref{rem:rational}.  Finally, although heaviness is crucial for our arguments we do not know whether the conclusion of the theorem holds without this assumption.

Corollary \ref{cor:cotangent} is a consequence of the above theorem combined with the fact that a tubular neighborhood of the zero section $T_0 \subset T^*\T^n$ can be symplectically embedded in $\CP^n$ (which is bounded) such that the image of $T_0$ in $\CP^n $ is a heavy Lagrangian submanifold -- the Clifford torus.

\begin{rem} \label{rem:first remark} A few remarks concerning 
Theorem \ref{thm:closing_near_heavy} are in order.  
\begin{itemize}
    \item[i.] Suppose that $L$ is a smooth Lagrangian.  Then,  $\Ham_L(M,\omega)$ contains the infinite dimensional group $\Diff_0(L)$, consisting of all diffeomorphisms of $L$ isotopic to identity, as a subgroup.  Indeed, it is a classical fact that  for any diffeomorphisms $f\in \Diff_0(L)$ there exists $\psi \in \Ham(M, \omega)$ which is supported in a (Weinstein) neighborhood of $L$ satisfying $\psi|_L = f$.  This provides a rich class of examples, some of which we examine more closely here. 

    \item[ii.] It is instructive to consider examples of $\psi \in \Ham_L(M,\omega)$ such that the dynamics of $\psi|_L$ is qualitatively very different than that of a conservative diffeomorphism. Observe that, in general, the closure of the periodic points of a generic $\psi'$ from Theorem \ref{thm:closing_near_heavy} intersects non-trivially with $L$. However, these periodic points do not necessarily accumulate at a periodic point on $L$. For instance, this is the case when the restriction of the original map $\psi|_L$ has no periodic points.
    
    \item[iii.] Another interesting scenario to consider is when $\psi|_L$ is uniquely ergodic.  In this case, the newly created periodic points have no choice but to converge to that unique invariant measure, yielding a statement akin to the ergodic closing lemma of Ma\~n\'e, albeit in a very simple setting.  It would be interesting to see whether one can design a smooth perturbation $\psi'$ that has periodic points approximating an a priori given (ergodic) measure of $\psi|_L$. 

    \item[iv.]  Suppose that $\psi|_L$ is minimal, i.e.\ every orbit of $\psi|_L$ is dense, for example, an  irrational rotation of a torus. In this case $L$ is contained in the closure of the periodic orbits of a generic perturbations $\psi'$  from Theorem \ref{thm:closing_near_heavy}. To see this, recall that, in general, the intersection between $L$ and the closure of the periodic points of a generic $\psi'$ from Theorem \ref{thm:closing_near_heavy} is non-empty. Then, since $\psi'\vert_L$ is  minimal, the intersection must be equal to $L$ because any closed invariant set which intersects $L$ must contain $L$.  We note that in this case, a statement in the spirit of the original $C^\infty$ closing lemma holds near $L$: For every open set $U$ that intersects $L$, there is a $C^\infty$-small perturbation of $\psi$ with a periodic point in $U$.
\end{itemize}
\end{rem}

Like most genericity statements, Theorem \ref{thm:closing_near_heavy} involves an application of the Baire category theorem.  The crucial ingredient of the argument, where Floer theory enters the picture, requires proving density of those elements in $\Ham_L(M,\omega)$ which have periodic orbits near $L$.  To that end, for a fixed $\psi \in \Ham_L(M,\omega)$, we consider a perturbation of the form $\varphi^s_G \psi$ where $G$ is a carefully chosen autonomous Hamiltonian.  On our way to proving the aforementioned density statement, we show that for a {\it full measure} set of parameters $s$, the map $\varphi^s_G \psi$ has periodic orbits satisfying the conclusion of Theorem \ref{thm:closing_near_heavy}.  

\begin{thm}\label{thm:closing_near_heavy-2}
Suppose that $(M, \omega)$ is rational and bounded, $L \subset M$ is heavy and that $\psi \in \Ham_L(M,\omega)$. Let $U$ be an open neighborhood of $L$.   There exists a Hamiltonian $G: M \rightarrow \R$, supported in $U$, such that $\varphi^s_G|_L = \mathrm{Id}$ and, moreover, for almost every $s \in [0,1]$ the composition $\varphi^s_G \psi$ possesses a sequence of periodic points converging to an invariant measure supported on $L$.  
\end{thm}

\subsection{Recurrent Lagrangians}
In this section, we present a general quantitative result, in the spirit of the closing lemma, which does not assume the invariance of the heavy set $L$ or boundedness of $(M, \omega)$. In the statement below $\gamma$ stands for the spectral pseudo-norm for the sub-additive spectral invariant for which $L$ is heavy; see Sections \ref{sec:spec-norm} and \ref{sec:heaviness}.  

We introduce some notation before stating our result. For a map $f \colon X \to X$ and two subsets $A, B \subset X$, we denote  the count of events $f^i(A) \subset B$ for $i \leq k$ by $T^k_f(A,B)$. In other words,
\begin{align}\label{eqn:event count}
T^k_f(A,B):= \# \{ 1 \leq i \leq k \, \vert \,  f^i(A) \subset B \} \leq k.
\end{align}
Note that if $A \subset B$ and $A$ is invariant under $f$, then $T^k_f(A,B) =k $. 

\begin{thm}
\label{thm:recurrent}
 Suppose that $(M, \omega)$ is rational and $L \subset (M, \omega)$ is heavy. Let $\psi \in \Ham(M,\omega)$, $V \subset M$ be an open neighborhood  of $L$ and $G \colon M \to \R$ be a non-negative autonomous Hamiltonian that is equal to its maximum on $V$, that is, $G\vert_V \equiv \max(G)$.  Then, for all $k\in \N$, there exists $s\in [0,1]$ such that the composition $ \varphi^s_{G}\psi$ has a $k$-periodic point $x$ satisfying 
    \[ 
       T_{\psi}^k(L, V) - T^k_{\varphi^s_{G}\psi} (x, \supp (G))\leq   \gamma(\psi^k)/ \max(G).
    \]
In particular, when $(M,\omega)$ is bounded, the error term $\gamma(\psi^k)/\max(G)$ can be chosen independent of $k \in \N$.
\end{thm}

Examining the statement above in the case where $(M,\omega)$ is bounded, we see that if $L$ is recurrent for $\psi$, with respect to the Hausdorff distance, then $C^\infty$-small perturbations of $\psi$ create periodic points which visit any neighborhood of $L$ with frequency bounded from below by the recurrence frequency of $L$.

In Section \ref{sec:pf_thm:precise1}, we prove Theorem \ref{thm:precise1} which generalizes both of Theorems \ref{thm:closing_near_heavy-2} \& \ref{thm:recurrent} and also provides a lower bound on the measure of the set of parameters $s\in [0,1]$ for which the above inequality holds.

\subsection{Application to dynamics of pseudo-rotations}
     
 Let $\psi$ be a pseudo-rotation of $(\C P^n, \omega_{FS})$, where $\omega_{FS}$ denotes the Fubini-Study symplectic form. In other words, we suppose that $\psi \in \Ham(\C P^n, \omega_{FS})$ has exactly $n+1$-periodic points, which are then necessarily fixed points.  It was shown in \cite{GG} that no fixed point of a pseudo-rotation is isolated as an invariant set. This is a higher dimensional analog of the celebrated Le Calvez--Yoccoz theorem \cite{LCY}. The invariant case of Theorem \ref{thm:recurrent} has an application of this nature.

\begin{cor}
\label{cor:pseudo-rotation}
A closed invariant heavy set $L \subset \CP^n \setminus \Fix(\psi)$ of a Hamiltonian pseudo-rotation $\psi$ of $\CP^n$ cannot be isolated as an invariant set of $\psi$. 
\end{cor}

In fact, the above corollary holds more generally for $\gamma$-rigid maps; see Remark \ref{rem:rigidity}. It is worth mentioning that pseudo-rotations can have complicated dynamics. For instance, there are pseudo-rotations $\psi \in \Ham (\CP^n, \omega_{FS})$ which are uniquely ergodic on the complement of $\Fix(\psi)$; see \cite{FK, LeS}. We believe the Anosov-Katok construction \cite{LeS} can be used to provide similar examples with invariant heavy sets as in Corollary \ref{cor:pseudo-rotation}.

\subsection*{Acknowledgments} 
S.S. is grateful to Patrice Le Calvez for numerous helpful conversations on the ergodic closing lemma, to Cheuk Yu Mak \& Ivan Smith for helpful comments related to Section \ref{subsec:FH}.

E.\c{C}. \& S.S. are partially supported by ERC Starting Grant number 851701.

S.T. was partially supported by a grant from the Institute for Advanced Study School of Mathematics,  the Schmidt Futures program,  a research grant from the Center for New Scientists at the Weizmann Institute of Science and Alon fellowship.

\section{Preliminaries}
\label{sec:prelims}
In this section we introduce some of our notations and conventions and recall the necessary background from symplectic geometry.

 \subsection{Invariant measures} \label{sec:measures}
  Let $\psi:X \to X$ be a continuous mapping of a compact topological space $X$. Denote by $\mathcal{M}_\psi$ the space of Borel probability measures on $X$ which are invariant under $\psi$.  We equip $\mathcal{M}_\psi$ with the weak topology which can be characterized as follows: a sequence of probability measures $\mu_n$ converges to $\mu$ if $\int f \mu_n \to \int f \mu$ for all continuous $f:X \to \R$.  In this topology, $\mathcal{M}_\psi$ is compact.
  
  A periodic point of $\psi$, say $x$, determines an invariant measure  $\mu_x^\psi \in \mathcal{M}_\psi$ via the formula 
    $$ \mu_x^{\psi} := \frac{1}{k} \sum_{i=0}^{k-1} \delta_{\psi^i(x)},$$
    where $k$ is the period of $x$ and $\delta_y$ denotes the Dirac delta measure at $y$. When it is clear from the context, we drop the sup-script $\psi$ from the notation. 

    When saying that a sequence of periodic orbits converges to a measure, as we do throughout this paper, we are referring to the convergence of the measures associated to the periodic points.  
 
  Suppose that $L \subset X$ is closed and invariant under $\psi$.  Here is a criterion, which we will use in our arguments, for convergence of a sequence of periodic orbits to an invariant measure supported on $L$.   Given a periodic point $x$ of $\psi$ and an open set $U$, the recurrence frequency with which $x$ visits $U$ is defined to be the quantity

  $$ \rho_\psi(x, U) = \frac{\# \{ 0 \leq i \leq k-1 \, \vert \,  \psi^i(x) \in U\}}{k}.$$
  Let $U_1 \supset U_2 \supset \cdots  U_n \supset \cdots$ be a nested sequence of  open neighborhoods of $L$ such that  $ \cap_n U_n = L$. 
  Let $x_n$ be a sequence of periodic points of $\psi$.  If 
  $\rho_\psi(x_n, U_n) \to 1$, then the limit points of the sequence $\mu_{x_n}$ are invariant measures supported on $L$.

 \subsection{Hamiltonian diffeomorphisms and related notions} \label{sec:Ham diffeos}

Let $(M^{2n}, \omega)$ be a closed and connected symplectic manifold. Throughout the paper we assume that $(M^{2n}, \omega)$ is rational, i.e.\ the image $\omega(\pi_2(M)) \subset \R$ is a discrete subset.

Let $C^\infty ( \R/ \Z \times M)$ be the set of smooth one-periodic Hamiltonians on $M$. The Hamiltonian vector field $X_H$ of $H \colon \R/ \Z \times M \to \R$ is defined  by $\omega(X_H, \cdot) :=-dH.$ The time-dependent flow  $\varphi_H^t$, or isotopy, generated by $X_H$ is called the Hamiltonian flow of $H$. We denote its time-one map by $\varphi_H$. Such  time-one maps form the group of Hamiltonian diffeomorphisms $\Ham(M,\omega)$ of $(M, \omega)$. It is a subgroup of the group of symplectomorphisms $\Symp(M, \omega)$ of $(M, \omega)$.

The support $\supp(H) \subset M$ of a Hamiltonian $H \in C^\infty ( \R/ \Z \times M)$ is given by
$$
\supp(H) := \overline{ \{ x \in M \ \vert \ H(x,t) \neq  0 \ \text{for some} \ t \in \R/\Z\}}.
$$
We denote by $\| \cdot \|_H$ the Hofer norm, or oscillation, on $ C^\infty( \R/ \Z \times M)$, namely
 \[
\|H\|_H:= \int_0^1 \max_x H(t,x) - \min_x H(t,x) \ dt. \]

Let $x \colon \R/ \Z \to M$ be a smooth contractible loop. A capping of $x$ is a map $v \colon D^2 \to M$ such that $v\vert_{\partial D^2}=x$. The pair $(x,v)$ is referred to as a capped loop and we say that two capped loops $(x_1, v_1), (x_2, v_2)$ are equivalent, if $x_1 = x_2$ and the sphere $v_1 \sharp (-v_2) \in \ker (\omega) \cap \ker (c_1)$ where $c_1$ is the first Chern class of $TM$.

Let $\Omega$ be the space formed by equivalence classes of capped loops in $M$.  Given a Hamiltonian $H$, the associated action functional  $\mathcal{A}_H: \Omega \rightarrow \R$ is defined by
\[
\mathcal{A}_H (x,v) := -\int_v \omega + \int_0^1 H(t, x(t)) dt.
\]
The set of critical points $\Omega_H \subset \Omega$ of the action functional $\mathcal{A}_H$ consist of capped loops $(x, v)$ such that $x$ is a contractible  one-periodic orbit of $X_H$.  The set of critical values 
\[
\Spec(H):= \{\mathcal{A}_H (x,v)\ \vert \  (x,v) \in \Omega_H \}
\]
is called the action spectrum of $H$. 

A one-periodic orbit $x \colon \R/ \Z \to M$ of $X_H$ is called non-degenerate if the linearized return map   $$D\varphi_H \colon T_{x(0)}M \to T_{x(0)}M$$  has no eigenvalues equal to one. We say that the Hamiltonian $H \in C^\infty ( \R/ \Z \times M)$ is non-degenerate if all one-periodic orbits of $X_H$ are non-degenerate. Finally, we denote by $\Fix(\varphi)$ the set of fixed points of $\varphi \in \Ham(M, \omega)$ and by $\Per(\varphi)$ the set of all periodic points, i.e., 
\[
\Per(\varphi) := \cup_{k \in \N} \Fix(\varphi^k). 
\]

\subsection{Remarks on Hamiltonians and related operations}\label{sec:hamiltonians}
Recall that the universal cover $  \tHam(M, \omega)$ of $\Ham(M,\omega)$ consists of the set of all Hamiltonian isotopies $\varphi^t_H, t\in [0,1]$, considered up to homotopy with fixed endpoints.  We denote by $\tV_H$ the natural lift of $\varphi_H$ (given by the isotopy $\varphi_H^t$)  to the universal cover. Up to reparametrization, every Hamiltonian isotopy can be written as the flow of a Hamiltonian $H$ that vanishes for values of $t$ close to $0, 1$.  Note that reparametrization does not affect the lift to the universal covering.  We will often consider Hamiltonians that vanish near $t=0,1$. This will be helpful in defining the following concatenation operation. 

The concatenation $H_1 \sharp \cdots \sharp H_k$ of the Hamiltonians $H_1, \ldots, H_k$ is defined by 
\begin{equation*}\label{eq:sharp}
H_1 \sharp \cdots \sharp H_k(t, x) := kH_i(k (t -(i-1)/k), x) \quad \text{for}\quad t \in [(i-1)/k, i/k].
\end{equation*}
Note that $H_i(k (t -(i-1)/k), x)=H_i(k t, x)$ since our Hamiltonians are one-periodic in time.  The Hamiltonian isotopy generated by the concatenation $H_1 \sharp \cdots \sharp H_k(t, x)$ is simply (a reparametrization of) the concatenation of the Hamiltonian flows of $H_1, \ldots, H_k$. In particular, we have $\varphi_{H_1 \sharp \cdots \sharp H_k } = \varphi_{H_k} \circ \cdots \circ \varphi_{H_1}$.  Furthermore, the two Hamiltonian isotopies  $\varphi^t_{H_1 \sharp \cdots \sharp H_k }$ and  $\varphi^t_{H_k} \circ \cdots \circ \varphi^t_{H_1}$ are homotopic relative endpoints, that is,
$$ \tV_{H_1 \sharp \cdots \sharp H_k } = \tV_{H_k} \circ \cdots \circ \tV_{H_1}.$$

The inverse $\bar{H}$ of the Hamiltonian $H$ is defined by 
\[
\bar{H}(t,x):=-H(t, \varphi^t_H(x)).
\]
Its flow $\varphi^t_{\bar{H}}$ is given by $(\varphi^t_H)^{-1}$ and hence  $\tV_{\bar H}=\tV_H^{-1}$. For $\psi \in \Symp(M, \omega)$, we define the pull back Hamiltonian $H\circ \psi$ by 
$$H\circ \psi (t, x) := H(t, \psi(x)).$$  The corresponding Hamiltonian flow is given by
$\varphi^t_{H\circ \psi} := \psi^{-1} \varphi^t_{H} \psi$. It is not difficult to check that when $\psi = \varphi_K$, we have
\begin{equation*}
\label{eq:conjugation}
  \tV_{H \circ \varphi_K} =  \tV_{\bar{K}} \, \tV_{H} \,  \tV_K.
\end{equation*}

\subsection{Sub-additive spectral invariants}
\label{subsec:spec_inv}
We start with an axiomatic definition of spectral invariants. 
The axioms listed here hold for Hamiltonian Floer homology spectral invariants associated to an idempotent quantum (co)homology class, in the setting of rational symplectic manifolds. Recall the notations $H\sharp G, \bar{H}, H\circ \psi,  \tV, \|\cdot\|_H$ from Section \ref{sec:hamiltonians}. 

\begin{definition}
    A functional 
    $$c:C^\infty(\R/\Z\times M)\rightarrow \R$$
    is called an \emph{sub-additive spectral invariant} if it satisfies the following properties:
    \begin{itemize}
        \item \textbf{Spectrality:} $c(H) \in \Spec(H)$ for all $H$.
        \item \textbf{Shift:} Let $a: \R/ \Z  \rightarrow \R$ be smooth.  Then, 
\begin{equation}
\label{eq:shift}
    c \big(H + a \big) = c\big(H  \big) + \int_{0}^1 a(t) \, dt.
\end{equation}

    \item \textbf{Hofer stability:} For any pair of Hamiltonians $G,H$ we have
    \begin{equation}\label{eq:Hofer_cont}
    \int_0^1 \min_x \Big(G(t,x)-H(t,x)\Big)\ dt\leq c(G)-c(H)\leq \int_0^1 \max_x \Big(G(t,x)-H(t,x)\Big)\ dt
    \end{equation}
    In particular, $c(\cdot)$ is continuous with respect to the Hofer norm $\|\cdot \|_H$, and is monotone: if $G\leq H$ then $c(G)\leq c(H)$.

    \item \textbf{Homotopy invariance:} Suppose that  $\tV_H=\tV_G$ in the universal cover $\tHam (M, \omega)$ and that  $\int_0^1\int_M H \omega^n dt = \int_0^1\int_M G \omega^n dt$.  Then,  
\begin{equation}
\label{eq:normal}
    c(H)= c\big(G\big).
\end{equation}

    \item \textbf{Sub-additivity:}   We have
\begin{equation}
\label{eq:tri}
    c( H\sharp G) \leq c (H) + c (G).
\end{equation}

    \item \textbf{Symplectic invariance:} For any $\psi \in \Symp_0(M, \omega)$ a symplectomorphism that is isotopic to the identity,  and $H \in C^\infty(\R/\Z \times M)$ we have 
\begin{equation}
\label{eq:symp_inv}
    c \big(H \circ \psi ) = c\big(H  \big).
\end{equation}

 \end{itemize}
\end{definition}

Notice that we do not assume the spectral invariants are normalized ($c(H\equiv 0)=0$) or that they are invariant for symplectomorphisms that are not isotopic to the identity.

\begin{rem}\label{rem:rational}
    Hamiltonian Floer homology spectral invariants satisfy the spectrality property when $(M, \omega)$ is rational.  Although the spectrality property is crucial to the proof of Lemma \ref{lem:derivative}, it seems plausible that it can be removed.  Without the rationality assumption, spectrality is known to hold only for non-degenerate Hamiltonians \cite{usher2008spectral, Oh09spectral}; this on its own is not sufficient for our purposes.  
\end{rem}

\subsection{The spectral pseudo-norm}\label{sec:spec-norm}
The spectral pseudo-norm $\gamma(H)$ of a Hamiltonian $H \colon \R/ \Z \times M \to \R$ is defined by 
\[
\gamma(H) :=c(H) + c(\bar H)
\]
and the spectral pseudo-norm $\gamma(\varphi)$ of $\varphi \in \Ham(M,\omega)$ is given by
\begin{equation}\label{eq:spec-norm}
    \gamma(\varphi) := \inf_{\varphi_H=\varphi} \gamma(H).
\end{equation}

The map  $$\gamma: \Ham(M, \omega) \rightarrow \R$$ is a conjugation invariant (possibly degenerate) pseudo-norm, i.e. it satisfies the following properties:
\begin{enumerate}
    \item \label{item:non-deg} $\gamma(\varphi) \geq 0$, with equality if $\varphi = \id$. 
    \item  $ \gamma(\varphi \psi) \leq \gamma(\varphi) + \gamma(\psi).$ 
    \item  $\gamma(\varphi \psi \varphi^{-1}) = \gamma(\psi)$.
\end{enumerate}

It induces a bi-invariant (possibly degenerate) metric $d_{\gamma}$ on $\Ham(M,\omega)$ by setting $$ d_{\gamma}(\varphi, \psi) :=\gamma(\varphi \psi^{-1}). $$  Recall that bi-invariance means $d_{\gamma}(\theta \varphi, \theta \psi) = d_{\gamma}( \varphi,  \psi) =  d_{\gamma}( \varphi \theta ,  \psi \theta)$.  

\medskip

We end this section with the definition of boundedness.

\begin{definition}\label{defn:bdd}
    We say a symplectic manifold $(M, \omega)$ is {\bf bounded} if it admits a an sub-additive spectral invariant $c$ such that the associated spectral pseudo-norm $\gamma$ is bounded. 
\end{definition}
\subsection{Heaviness}\label{sec:heaviness}
The homogenized spectral invariant is defined \cite{entov2006quasi} to be 
\[
\zeta(H):= \lim_{k\rightarrow \infty}\frac{c(H^{\sharp k})}{k}.
\]
Notice that by the sub-additivity property, $\zeta(H)\leq c(H)$.
A subset $L\subset M$ is \emph{heavy} \cite{entov2009rigid} with respect to a spectral invariant $c$ if 
\[
\inf_{x\in L} H(x) \leq \zeta(H) \quad\text{for all time-independent }H:M\rightarrow\R.
\]
Note that the heaviness of a given set $L$ depends on the spectral invariant, see for example \cite{entov2009rigid,kawamoto2024spectral}. 

Typically, heavy sets come from certain Lagrangian submanifolds or CW complexes, with non-vanishing Lagrangian Floer homology. Heaviness implies the following bound for time-dependent Hamiltonians.
\begin{lemma}\label{lem:heavy-time-dependent}
    Let $G:\R/\Z\times M\rightarrow \R$ be a time-dependent Hamiltonian, then
    \[
    c(G)\geq \int_0^1 \inf_{x\in L} G(t,x)\ dt.
    \]
\end{lemma}
\begin{proof}
    Denote $a(t):=\inf_{x\in L} G(t,x)$ then we need to show that $c(G)\geq \int_0^1 a(t)\ dt$. For that end, set $H(x):=\min_t \Big(G(t,x)-a(t)\Big)$, then clearly $H$ is autonomous and therefore 
    \[
    c(H)\geq \zeta(H) \geq \inf_{x\in L} H(x). 
    \]
    Since $H(x)\leq G(t,x)-a(t)$ pointwise, the stability and shift properties of $c$ imply
    \[
    c(H)\leq c(G-a) = c(G) - \int_0^1 a(t)\ dt.
    \]
    To prove the claim, it suffices to show that $0 \leq c(H)$.  This last inequality is a consequence of the fact that $\inf_{x\in L} H(x)\geq 0$, whose proof we omit. 
\end{proof}

\subsection{Examples} \label{subsec:FH}
Many of the results in our paper require pairs $(M, L)$, with $M$ a rational and bounded symplectic manifold and $L$ a heavy subset of $M$.  It is important to note that  we require the existence of a sub-additive spectral invariant for which (the rational) $M$ is bounded and $L$ is heavy. We see in this section that examples of such pairs are abundant. 

All examples of sub-additive spectral invariants known to us arise from Hamiltonian Floer homology and the fact that it is isomorphic to quantum homology.  Every idempotent $e$ of the quantum homology ring of $(M, \omega)$, gives rise to a sub-additive spectral invariant satisfying the above axioms \cite{schwarz, oh, McDuff-Salamon-Jcurves, entov2009rigid}.

It is well known that the spectral norm\footnote{In the case of the fundamental class, $\gamma$ is non-degenerate, hence a genuine norm.} associated to the unit element, i.e.\ the fundamental class, in Floer homology (with field coefficients) is bounded on $\CP^n$  \cite{entov2006quasi,polterovich2014function}. Furthermore, products of complex projective spaces \cite[Thm.\ 6.1] {entov2006quasi}, certain other toric manifolds \cite{OT-toric}, and one point blow-ups of a large class of symplectic manifolds \cite[Thm.\ 7.1] {McDuff-uniruled, entov2006quasi}  all admit a bounded spectral pseudo-norm. In all of these cases, the underlying idempotent is the unit element of a field direct summand of the quantum homology.\footnote{As a consequence, all heavy sets for these idempotents are super-heavy \cite[Rem.\ 1.3]{entov2009rigid}.} As for the corresponding heavy sets, the Clifford torus and $\RP^n$ in $\C P^n$ are both heavy, as well as the Clifford torus in $\C P^1\times \C P^1$ \cite{entov2009rigid}; for (monotone) toric symplectic manifolds, a certain fiber of the moment map, called the special fiber, is heavy for any idempotent \cite[Thm.\ 1.11]{entov2009rigid}. In general, for any closed symplectic manifold and for any idempotent, the co-dimension one skeleton of a sufficiently fine triangulation is heavy \cite[Thm.\ 1.6]{entov2009rigid}. 

 Usher and Fukaya-Oh-Ohta-Ono have shown that bulk deformed versions of Hamiltonian Floer homology produce sub-additive spectral invariants with bounded spectral pseudo-norm on a large class of symplectic manifolds, including all toric symplectic manifolds; see \cite[Thm.\ 1.4, Thm.\ 1.6]{usher2011deformed} and \cite[Thm.\ 1.1, Cor.\ 1.2, Thm.\ 16.3, Lem.\ 16.5]{fukaya2019spectral} .  However, we have chosen to avoid using bulk-deformed Hamiltonian Floer homology in order to sidestep complications associated with the spectrality property of bulk-deformed spectral invariants. Finally, it seems that the results in \cite{judd2019}, combined with the theory developed in \cite{fukaya2019spectral}, yield sub-additive spectral invariants with bounded spectral pseudo-norm for all rational toric symplectic manifolds.\footnote{We thank Cheuk Yu Mak for pointing this out to us.}

\subsection{Heaviness in Herman's non-closing lemma}\label{sec:Herman-heaviness}
In this section we use Herman's counterexample to the $C^k$ closing lemma to show that the boundedness assumption in our results is necessary. In this section, $M=\T^4$ is the 4-dimensional torus with coordinates $x_1,x_2,x_3,x_4\in \R/\Z$. Let 
$$\omega_\alpha:= dx_2\wedge dx_1 + dx_4\wedge dx_3 + \alpha_1 dx_3\wedge dx_2+ \alpha_2 dx_1\wedge dx_3$$ 
be the symplectic form  used by Herman\footnote{This symplectic form was originally considered by Zehnder \cite{Zehnder-torus}.}, corresponding to a vector $\alpha=(\alpha_1,\alpha_2,1)\in \R^3$ with rationally independent entries and satisfying a certain Diophantine condition, see \cite{herman1991exemples} and \cite[Section 4.5]{hofer2012symplectic} for further details.
Consider the Hamiltonian $H(x)=\beta(x_4)$ for $\beta:[0,1]\rightarrow [0,1]$ that is periodic and is linear on an open interval $I\subset (0,1)$. For this symplectic form, the Hamiltonian $H(x)$ generates the Hamiltonian vector field
$$
X_{H} = \beta'(x_4)\cdot (\alpha_1, \alpha_2,1,0)
$$
which is constant for $x_4\in I$. The flow on $\T^3\times I\subset \T^4$ is an irrational translation on each $\{x_4=\theta\}\cong\T^3\times \{\theta\}$, and thus $\varphi_{H}$ has no periodic points in the open set $\T^3\times I$. By Herman's non-closing lemma (e.g.\ Theorem 10 in \cite[Section 4.5]{hofer2012symplectic}), for an open $I'$ compactly contained in $I$, there is an open $C^\infty$-neighborhood of $H$ in $C^\infty(\T^4;\R)$ consisting of autonomous Hamiltonians with no periodic points in $\T^3\times I'$.

Fix $\theta \in I'$ and denote $L:=\T^3\times \{\theta\}$.  Clearly, $\varphi_H \in \Ham_L(M, \omega_\alpha)$ but the conclusion of Theorem \ref{thm:closing_near_heavy-2} does not hold here.  In Lemma \ref{lem:T3_heavy} below, we show that the hypersurface \( L := \T^3 \times \{\theta\} \) is a heavy set in \( (\T^4, \omega_\alpha) \) with respect to the spectral invariant associated with the fundamental class in Hamiltonian Floer homology. This implies that the failure of Theorem \ref{thm:closing_near_heavy-2} in this setting is due to the lack of boundedness.\footnote{We note that while this lack of boundedness follows from the above reasoning, it can also be established directly using an argument similar to the proof of Lemma \ref{lem:T3_heavy}.}

\medskip
In the lemma below, $$c: C^\infty([0,1] \times M) \rightarrow \R$$
denotes the Hamiltonian spectral invariant associated to the fundamental class.

\begin{lemma}\label{lem:T3_heavy}
    The set $L:=\T^3\times \{\theta\}$  is heavy  in $(\T^4, \omega_\alpha)$ with respect to the sub-additive spectral invariant $c$.
\end{lemma}
\begin{proof}
    By the monotonicity property of spectral invariants, it is sufficient  to show that for any autonomous Hamiltonian $G$  that is constant on $L$, $G|_L\equiv a$, it holds that $c(G)\geq a$. 

    Given such Hamiltonian $G$, there exists an autonomous $F$ such that 
    \begin{enumerate}
        \item $F\leq G$,
        \item $F|_L=G|_L$,
        \item $F = f(x_4)$ only depends on $x_4$,
        \item $F$ has exactly two critical values: 0 and $a$.
    \end{enumerate}
    For $F$ as above, $X_F = f'\cdot (\alpha_1, \alpha_2,1,0)$. In particular the flow has no non-constant one periodic orbits. It therefore follows that the action spectrum is $\Spec(F)=\{0,a\}$. Since $c$ is the spectral invariant associated to the fundamental class, we must have $c(F)=a$. By the monotonicity property of the spectral invariants, this implies $c(G)\geq a$.
\end{proof}

\section{General and quantitative results}

In this section, we state Theorem \ref{thm:precise1} which generalizes Theorem \ref{thm:recurrent} and plays an important role in the proofs of our main results.

Before stating our theorem, we describe the setting in which it applies.  Let $(M, \omega)$ be closed and connected and fix an sub-additive spectral invariant $$ c: C^\infty([0,1] \times M) \rightarrow \R. $$
The choice of the spectral invariant $c$ is fixed for the remainder of the paper, unless otherwise stated. 
Hence, when we say $(M, \omega)$ is bounded we mean that it is bounded, in the sense of Definition \ref{defn:bdd}, with respect to the spectral pseudo-metric $\gamma$ determined by $c$.  Similarly, when we say a subset $L \subset M$ is heavy, we mean that it is heavy, in the sense of Section \ref{sec:heaviness}, with respect to $c$.

    Given $\psi \in \Ham(M, \omega)$, and subsets $L, V\subset M$, we define the recurrence frequency  
    \[
    \rho_{\psi}(L, V):=\limsup_k  \frac{T_{\psi}^k(L, V)}{k},
    \]
    where $T_{\psi}^k(L, V)$ is defined as in \eqref{eqn:event count}. Recall from Section \ref{sec:measures} that for a $k$-periodic point $x$ of $\psi$, the associated invariant probability measure $\mu_x^{\psi}$ is given by 
\[
\mu_x^\psi =\frac{1}{k}\sum_{j=0}^{k-1} \delta_{\psi^j(x)}.
\]

\begin{thm}
\label{thm:precise1}
    Suppose that $(M, \omega)$ is rational and that $L \subset (M, \omega)$ is heavy. Let $\psi \in \Ham(M,\omega)$,  and let $G \colon M \to \R$ be a non-negative autonomous Hamiltonian that is equal to its maximum on $L$. Denote by $V:=G^{-1}(\max (G))$. We have:
    \begin{itemize}
        \item[(i)] For all $k\in \N$, there exists $s\in [0,1]$ such that the composition $ \varphi^s_{G}\psi$ has a $k$-periodic point $x$ which satisfies 
    \begin{equation}
    \label{eq:thm_precise1_i}
        \frac{1}{\max(G)}\int_M G \ d\mu^{\varphi^s_{G}\psi}_ x\geq\frac{1}{k} T_{\psi}^k(L, V) - \frac{\gamma(\psi^k)}{k\cdot \max(G)}
    \end{equation}
   
    \item[(ii)] Suppose in addition that $(M, \omega)$ is bounded. Then, for all $0< \alpha <1$, there exists a sequence of subsets $S_k^{\alpha}\subset [0,1]$  with 
    \begin{equation}\label{eq:measure_bound}
        \limsup_{k \to \infty} m(S_k^{\alpha}) \geq \frac{\alpha \rho_{\psi}(L, V)}{1-(1-\alpha)\rho_{\psi}(L, V)}  
    \end{equation}
     and such that for all $s \in S_k^{\alpha}$, the composition $\varphi^s_{G}\psi$ has a $k$-periodic point $x$  satisfying
        \begin{equation}
        \label{eq:thm_precise1_ii}
           \frac{1}{\max(G)}\int_M G \ d\mu^{\varphi^s_{G}\psi}_x > (1-\alpha) \frac{1}{k} T_{\psi}^k(L, V).
        \end{equation}
    \end{itemize} 
\end{thm}

Note that, since $G\vert_L = \max(G)$, we have $\varphi_G^s|_L = \mathrm{Id}$. In particular, when $L$ is invariant under $\psi$, the composition $ \varphi^s_{G}\psi \in \Ham_{L, \psi}(M,\omega)$.

\begin{rem}
    \label{rem:lim_vs_limsup}
    It is easy to see from the proof of the Theorem \ref{thm:precise1} that when the limit  $T_{\psi}^k(L, V)/k$ exists, one can  replace $\limsup$ with $\liminf$ in part (ii) of Theorem \ref{thm:precise1}; see Corollary \ref{cor:precise_for_invariant}. 
\end{rem}

\begin{rem}  
\label{rem:thm_precise1}
We record the following observations about the statement of Theorem \ref{thm:precise1}.

    \begin{enumerate}
        \item Theorem \ref{thm:precise1} gives a lower bound for the number of times the orbit of the periodic point $x$ visits the support of $G$:
        \begin{align*}
            \frac{1}{\max(G)}\int_M G \ d\mu^{\varphi_G^s\psi}_x =  \frac{1}{k} \sum_{i=0}^{k-1} \frac{G((\varphi_G^s\psi)^i(x))}{\max(G)} \leq
            \frac{1}{k}T^k_{\varphi_{G}^s\psi}(x,\supp(G)).
        \end{align*}
         
        \item  Part (i) of Theorem \ref{thm:precise1} is only interesting for iterates which satisfy 
        $$T_{\psi}^k(L,V) > \gamma(\psi^k)/\max(G).$$
        In this case, \eqref{eq:thm_precise1_i} implies that $T^k_{\varphi^s_{G}\psi} (x, \supp (G)) >0$ and hence, in words, the perturbation $\varphi^s_{G}\psi$ has a periodic point passing through the support of $G$. 
        Otherwise, when $T_{\psi}^k(L,V) \leq \gamma(\psi^k)/\max(G)$, the conclusion \eqref{eq:thm_precise1_i} automatically holds for any $s \in [0,1]$ and any $k$-periodic point $x$ of $\varphi^s_{G}\psi$, regardless of whether or not the orbit of $x$ passes through the support of $G$. 

        \item   Similarly, in part (ii), if $\rho_{\psi}(L, V)=0$, the statement holds trivially for $S_k=\emptyset$. However, if $\rho_{\psi}(L, V)=0$  and we additionally have $T_{\psi}^k(L,V) \to \infty$ as $k\to \infty$, then part (i) tells us that for any fixed $0<\alpha<1$, there exist non-empty $S_k \subset [0,1]$ satisfying \eqref{eq:thm_precise1_ii} for all sufficiently large $k$. 
    \end{enumerate}
\end{rem}

Theorem \ref{thm:precise1} is proved in Section \ref{sec:pf_thm:precise1}. The following is an immediate corollary of Theorem \ref{thm:precise1} and Remarks \ref{rem:lim_vs_limsup}, \ref{rem:thm_precise1}. 

\begin{cor}\label{cor:precise_for_invariant}
Suppose that $(M, \omega)$ is rational, $L \subset (M, \omega)$ is heavy and  $\psi \in \Ham(M, \omega) $ leaves $L$ invariant. Let $G \colon M \to \R$ be a non-negative autonomous Hamiltonian that is equal to its maximum on $L$. We have:
\begin{itemize}

   \item[(i)] For all $k \in \N$, there exists $s\in [0,1]$ such that the composition $\varphi^s_{G}\psi$ has a $k$-periodic point $x$ which satisfies 
     $$
     T^k_{\varphi^s_{G}\psi} (x, \supp (G)) \geq k - \gamma(\psi^k)/\max(G).
   $$
   In particular, when $(M,\omega)$ is bounded, the error term $\gamma(\psi^k)/\max(G)$ can be chosen independent of $k \in N$.
 
    \item[(ii)] Suppose in addition that $(M, \omega)$ is bounded. Then, for all $0< \alpha <1$, there exists $S_k^\alpha \subset [0,1]$ with $m(S_k^\alpha) \to 1$ as $k \to \infty$ and such that for all  $s \in S_k^\alpha$, the composition $\varphi^s_{G}\psi$ has a $k$-periodic point $x$ satisfying 
    \begin{equation}\label{eq:cor_precise2}
        \frac{1}{\max(G)}\int_M G\ d\mu^{\varphi^s_{G}\psi}_x > (1-\alpha).
    \end{equation}
\end{itemize}
\end{cor}

\section{Proofs of the main results}

In this section we prove our main results, Theorems \ref{thm:closing_near_heavy}, \ref{thm:closing_near_heavy-2} and Corollary \ref{cor:pseudo-rotation}, assuming Theorem \ref{thm:precise1} (and Corollary \ref{cor:precise_for_invariant}). Note that Theorem \ref{thm:recurrent} directly follows from part (i) of Theorem \ref{thm:precise1} and Remark \ref{rem:thm_precise1}. 

\subsection{Proof of Theorem \ref{thm:closing_near_heavy}}

The two parts of the theorem have very similar proofs.  We provide here a proof of part (ii).

Note that if $\psi$ has a periodic point on $L$, then the conclusion holds for all elements of $\Ham_{L, \psi}(M,\omega)$.  Therefore, we will assume, for the rest of the proof, that $\psi$ has no periodic points on $L$.

Let $U_1 \supset U_2 \supset \cdots  U_n \supset \cdots$ be a nested sequence of  open neighborhoods of $L$ such that  $ \cap_n U_n = L$.  

Let $\mathcal{G}_n \subset \Ham_{L, \psi}(M,\omega)$ consist of all $\phi \in \Ham_{L, \psi}(M,\omega)$ that have a non-degenerate periodic point, say $x$, satisfying 
\begin{equation*}
                \rho_\phi (x, U_n) > 1- \frac{1}{n}.
\end{equation*}
The set $\mathcal{G}_n$ is open in $\Ham_{L, \psi}(M,\omega)$ because the existence of a point $x$ as above is an open condition. We will show below that the set $\mathcal{G}_n$ is also  $C^\infty$ dense in $\Ham_{L, \psi}(M,\omega)$. 
Assuming the density of $\mathcal{G}_n$, consider $\mathcal{G}:= \cap _n \mathcal{G}_n.$  The set $\mathcal{G}$ is dense by the Baire category theorem and, moreover, every $\phi \in \mathcal{G}$ has periodic points converging to an invariant measure supported on $L$; see the convergence criterion in Section \ref{sec:measures}.

It remains to show that $\mathcal{G}_n$ is dense in $\Ham_{L, \psi}(M,\omega)$. Given $\phi \in \Ham_{L, \psi}(M,\omega) \setminus \, \mathcal{G}_n$, by part (i) of Corollary \ref{cor:precise_for_invariant}, we can find  $\varphi_1$,  arbitrarily $C^\infty$ small, coinciding with the identity on $L$ and such that $\phi' = \varphi_1 \phi$  has a periodic point $x$ satisfying 
            \begin{equation*}
                \rho_{\phi'} (x, U_n) > 1- \frac{1}{n}.
            \end{equation*} 
If $x$ is degenerate, then we can pick an arbitrarily $C^\infty$-small $\varphi_2$ such that $x$ becomes a non-degenerate periodic point of $\phi'' := \varphi_2 \phi'$  and, moreover, the orbit of $x$ under $\phi''$ coincides with the orbit of $x$ under $\phi'$.   This can be achieved by picking $\varphi_2$ to fix $x$ and be supported in a small neighborhood $V$ of $x$; if $V$ is sufficiently small so that the sets $V, \phi'(V), \ldots, \phi'^{k-1}(V)$ are pairwise disjoint, then $x$ will have the same orbit under $\phi'$ and $\phi''$.  Note that $\phi'' \in \Ham_{L, \psi}(M,\omega)$ because $\varphi_2$ is supported away from $L$.  Furthermore, since the orbit of $x$ under $\phi''$ coincides with its orbit under $\phi'$, we have $\rho_{\phi''} (x, U_n)=\rho_{\phi'} (x, U_n)> 1-1/n$. This completes the proof of density of $\mathcal{G}_n$.

\subsection{Proof of Theorem \ref{thm:closing_near_heavy-2}}

    We need to construct a Hamiltonian $G:M\rightarrow\R$ supported in $U$, that satisfies $\varphi_G^s|_L=\operatorname{Id}$ for all $s \in [0,1]$, and such that for almost every $s\in [0,1]$, the composition $\varphi_G^s\psi$ has a sequence of periodic points converging to an invariant measure supported on $L$; see Section \ref{sec:measures}.

    We will use part (ii) of Corollary \ref{cor:precise_for_invariant}. Let $G:M\rightarrow\R$ be a non-negative Hamiltonian such that $G^{-1}(\max(G))= L$, in other words, $G$ coincides with its maximum on $L$ and is strictly smaller elsewhere. For each $\alpha\in(0,1)$, we have a full measure set
    \[
    S^\alpha:= \cup_{k\in \N} S^\alpha_k \subset[0,1]
    \]
    with the property that for all $s\in S^\alpha$, the composition $\varphi^s_G \psi$ has a periodic point satisfying \eqref{eq:cor_precise2}. Consider the sequence $\alpha_j:=1/j$ and set 
    \[
    S:= \cap_{j=1}^\infty S^{\alpha_j}.
    \]
    Then $m(S)=1$ as well, since it is a countable intersection of full measure sets. 
    
    Fix $s\in S$, and for ${\alpha_j}=1/j$, let $x_j$ be a periodic point of $\varphi_G^s \psi$ guaranteed by part (ii) of Corollary \ref{cor:precise_for_invariant} (here we are using the fact that $S\subset S^{\alpha_j}$ for all $j$). We claim that all limit points of the sequence
    \[
    \mu^{\varphi_G^s\psi}_{x_j}
    \]
    are supported on $L$. Indeed, let $\mu$ be a limit point. Taking the limit over the inequality \eqref{eq:cor_precise2}, we see that
    \[
    \int_M \frac{G}{\max(G)} \ d\mu = 1.
    \]
    Since $G^{-1}(\max(G))= L$, this is only possible if $\mu$ is supported on $L$.

\subsection{Proof of Corollary \ref{cor:pseudo-rotation}}
\label{sec:pf_pseudo}

Let $\psi \in \Ham(\CP^n, \omega_{FS})$ be a pseudo-rotation. For a contradiction, suppose that there exists a closed invariant heavy set $L \subset \CP^n \setminus \Fix(\psi)$ that is isolated as an invariant set. Fix an isolating neighborhood $V$ of $L$ disjoint from $\Fix(\psi)$. That is, we fix an open set $V \subset \CP^n \setminus \Fix(\psi)$ such that $L$ is the maximal, with respect to inclusion, invariant set of $\psi$ in $V$.

Next choose two open sets $U_1$, $U_2 \subset \CP^n$ with the  property
\[
L \subset U_1 \subset \overline{U_1} \subset U_2 \subset \overline{U_2} \subset V
\]
and let $G$ be a non-negative (and non-zero) autonomous Hamiltonian supported in $V$, equal to its maximum on $U_1$, and such that the derivative $dG$ is supported in $U_2 -\overline{U_1}$. 

Set $W=\overline{U_2} -U_1 \subset V$. By definition of $V$, for all $x \in W$ there exists $k_x \in \Z$ such that $\psi^{k_x}(x) \notin V$. Since $W$ is compact, there exists $K \in \N$, a uniform exit time, such that $k_x$'s above can be chosen to satisfy $|k_x| < K$ for all $x \in W$. Observe that this implies that a $C^0$-small perturbation of $\psi$ cannot have a periodic point passing through $W$ and whose entire orbit is  contained in $V$. Fix a small $\epsilon >0$ so that this property holds for the composition $\varphi^s_{\epsilon G} \psi$ for all $s\in [0,1]$.

Recall that, see \cite{GG, JS}, all pseudo-rotation of $\psi$ of $(\C P^n, \omega_{FS})$ are rigid with respect to the standard spectral norm $\gamma$, in other words, there exits a sequence of iterations $k_i \to \infty$ such that $\gamma(\psi^{k_i}) \to 0$. In particular, one can find $k \in \N$ which satisfies
$$
\gamma(\psi^{k})/\epsilon\max(G) <1/2.
$$
Then part (i) of Theorem \ref{thm:precise1} implies that there exist $s \in [0,1]$ such that the composition $\varphi^s_{\epsilon G} \psi$ has a $k$-periodic point whose orbit is entirely contained in $\supp(G) \subset V$. On the other hand, since the derivative $dG$ is supported in $U_2 -\overline{U_1} \subset W $, all newly created periodic points of the perturbation $\varphi^s_{\epsilon G} \psi$ should pass through $W$. This is a contradiction. 

\begin{rem}\label{rem:rigidity}[$\gamma$-rigidity]
The proof of Corollary \ref{cor:pseudo-rotation} word-by-word applies to any $\psi \in \Ham(M,\omega)$ which is rigid with respect to some spectral pseudo-norm: Such maps do not admit an isolated closed heavy invariant set $L \subset M\setminus \overline{\Per(\psi)}$. We restricted ourselves to pseudo-rotation of $(\C P^n, \omega_{FS})$ for simplicity. It is worth mentioning that, examples of rigid maps are very rare. In the case of the standard spectral norm, it is known that $C^\infty$-generically they don't exists; see \cite{CGG-spectral}. We don't know whether the same holds for a general spectral pseudo-norm as well. 
\end{rem}

\section{Proof of Theorem \ref{thm:precise1}}
\label{sec:pf_thm:precise1}
We present in this section our proof of Theorem \ref{thm:precise1}.  We assume throughout that $(M, \omega)$ is rational.

\subsection{Preliminary Lemmas}

The proof of Theorem \ref{thm:precise1} relies on the following three lemmas. We note that the rationally assumption on $(M,\omega)$ is only used in Lemma \ref{lem:derivative} and the heaviness of $L$ in Lemma \ref{lem:heavy}. 

\begin{lemma}
\label{lem:derivative}
    For any smooth family of Hamiltonians $H_s \colon \R /\Z \times M \to \R$  where $s \in [0,1]$, we have:
    \begin{itemize}
        \item[(i)] There exists $s\in [0,1]$ and a one-periodic orbit $x_s \colon \R /\Z \to M$ of the Hamiltonian $H_s$ such that  
        \begin{equation}
        \label{eq:spec_difference_leq_dHs}
            c(H_1) - c(H_0) \leq \int_{x_s} \partial_s H_s . 
        \end{equation}
        
        \item[(ii)] For $\eta>0$, let $S_{\eta}$ be the set of parameters $s \in [0,1]$ such that the inequality
         \begin{equation}
         \label{eq:ineq_forall_s_in_S_eta}
         c(H_1) - c(H_0) \leq \int_{x_s} \partial_s H_s +\eta
        \end{equation}
        holds  for some one-periodic orbit $x_s \colon \R /\Z \to M$ of $H_s$. We have 
        \begin{equation*}
       m(S_{\eta}) \geq \frac{\eta}{K-\Delta C + \eta} 
        \end{equation*}
        where $\Delta C := c(H_1) - c(H_0)$ and  $K := \max_{s,t,x} \partial_s H_s(t,x) $. 
        
    \end{itemize}
\end{lemma}

We remark here that the proof of the first item in the above lemma requires computing the derivative of the map $s\mapsto c(H_s)$.  A similar strategy is applied in \cite[Lemma 3.4]{Irie-equidist} and \cite[Lemma 3.3]{prasad2022}.

\begin{proof}  
    We first prove the lemma for a generic path and then show that the general case follows by continuity. We argue as in \cite[Lemma 21]{HLeS}. Let $X\subset [0,1]$ be the set of parameters for which $H_s$ is degenerate or has two periodic orbits of the same action (such paths are referred to as admissible in \cite{HLeS}). For a generic path, $X$ is a finite set (here we use the assumption that $(M, \omega)$ is rational). We claim that in this case
\begin{equation}
\label{eq:spec_derivative}
     c(H_1)-c(H_0) = \int_{[0,1]\setminus X} \int_{x_s}  \partial_s H_s \ ds
\end{equation}
    for some one-periodic orbits $x_s$ of $H_s$. Indeed, for every $s_0\in [0,1]\setminus X$, one can find a (smooth) one-parameter family $\{(x_s, v_s)\}_{s\in[s_0-\epsilon, s_0+\epsilon]}$ of capped one-periodic orbits of $H_s$ such that $\mathcal{A}_{H_s}(x_s, v_s) = c(H_s)$. It follows that the function $s\mapsto c(H_s)$ is smooth on $[0,1] \setminus X$. Since $X$ is finite, we have
\[
    c(H_1)-c(H_0) = \int_{[0,1]\setminus X} \frac{d}{ds} c(H_s) \ ds. 
\]
    Let us compute the derivative $\frac{d}{ds} c(H_s)$ at $s_0\in [0,1]\setminus X$ to see that it agrees with the right hand side of \eqref{eq:spec_derivative}:
        \begin{align*}
        \frac{d}{ds} c(H_s)|_{s_0}&=\frac{d}{ds} \cA_{H_s} (x_s,v_s)|_{s_0} = \frac{d}{ds} \cA_{H_s} (x_{s_0},v_{s_0})|_{s_0} +  \frac{d}{ds} \cA_{H_{s_0}} (x_s,v_s) |_{s_0}\nonumber \\
        &=\frac{d}{ds} \Big(\int_{x_{s_0}} H_s - \int_{v_{s_0}} \omega\Big)\Big|_{s_0} + d\cA_{H_{s_0}}(x_{s_0},v_{s_0}) \nonumber\\
        &= \int_{x_{s_0}} \partial_s H_s|_{s_0} - 0+ 0
        \end{align*}
        where in the last equality we used the fact that the capped periodic orbit $(x_{s_0} , v_{s_0})$ is a critical point of $\cA_{H_{s_0}}$ and that the integral $ \int_{v_{s_0}} \omega$ is independent of $s$. 
        
        We now turn to the proofs of items (i) and (ii) for a generic path. Observe that (i) immediately follows from \eqref{eq:spec_derivative}. For (ii), using \eqref{eq:spec_derivative} again, we write \begin{align*}
            \Delta C \leq \int_{S_\eta}\int_{x_s} K \ ds +\int_{[0,1]\setminus (S_\eta\cup X)} \Delta C-\eta \ ds,
        \end{align*}
        where $K = \max_{s,t,x} \partial_s H_s(t,x)$, from which $ m(S_{\eta}) \geq \frac{\eta}{K-\Delta C + \eta}$ follows.

        Next we prove the general case. Let $H_s^j$ be a sequence of paths converging to $H_s$, that is, 
    \[
    H_s^j\xrightarrow{C^\infty}H_s
    \]
    uniformly in $s$. Suppose that the lemma holds for $H_s^j$ for all $j$. We will show that then it holds for $H_s$ as well. First note that, by continuity of spectral invariants, we have
        \[
        \Delta C_j \to \Delta C
        \]
        where $\Delta C_j:= c(H_1^j)-c(H_0^j)$. Secondly, by assumption, for each $j$, there exist $s_j \in [0,1]$ and a one-periodic orbit $x_{s_j}^j$ of $H^j_{s_j}$ that satisfies \eqref{eq:spec_difference_leq_dHs}. Then, up to passing to a subsequence, one can find $s\in [0,1]$ such that 
    \[
    s_j\xrightarrow[j\rightarrow\infty]{} s \quad\text{and}\quad H_{s_j}^j\xrightarrow[j\rightarrow \infty]{C^\infty} H_s.
    \]
    By Arzel\`a-Ascoli theorem, the orbits $x_{s_j}^j$  converge (up to a subsequence) to some one-periodic orbit $x_s$ of $H_s$. Observe that, by continuity, the limit $x_s$ satisfies \eqref{eq:spec_difference_leq_dHs} as well. 
    We conclude that (i) holds for an arbitrary path. 
   
    Regarding (ii), for each $j$, let $S_\eta^j \subset [0,1]$ be the set given by  \eqref{eq:ineq_forall_s_in_S_eta}. Define $S_\eta' \subset [0,1]$ as the set of  all limit points of the sequences $s_{j}\in S_\eta^{j}$.
   Observe that $S_\eta'$ is closed and every neighborhood of it contains $S_\eta^j$ for all sufficiently large $j$. It follows that
    \[
    \frac{\eta}{K-\Delta C+\eta}=\lim_{j\rightarrow \infty}\frac{\eta}{K_j-\Delta C_j+\eta}\leq \limsup_{j\rightarrow\infty} m(S_\eta^j) \leq m(S_{\eta}')
    \]
    where $K_j:= \max_{s,t,x} \partial_s H_s^j(t,x)$. It remains to show that $S_{\eta}' \subset S_{\eta}$. We argue as in part (i). For $s\in S_\eta'$, take $S_\eta^{j_k}\ni s_{j_k}\rightarrow s$ and periodic orbits $x_{s_{j_k}}^{j_k}$ of $H_{s_{j_k}}^{j_k}$ for which \eqref{eq:ineq_forall_s_in_S_eta} holds. After passing to a subsequence, they converge to a periodic orbit $x_s$ of $H_s$ which, by continuity, satisfies  \eqref{eq:ineq_forall_s_in_S_eta}. This completes the proof of the lemma. 
\end{proof}

\begin{setup}
\label{set:setup_for_lemmas}
    The next two lemmas are in the following setting:
\begin{itemize}
    \item $V\subset M$ is a fixed subset containing the heavy set $L$ from the statement of Theorem \ref{thm:precise1}. 
    
    \item $G:[0,1]\times M\rightarrow \R$ is a time-dependent non-negative Hamiltonian that satisfies:
    \begin{enumerate}
        \item $G(t,x)=0$ for $t$ near $0,1$ (this is to guarantee that the $\sharp$ operation is smooth).
        \item for every $t$, $G(t,-)$ coincides with its maximum on $V$ and is strictly smaller elsewhere:
        \[\text{for all } t\in [0,1],\quad G(t,x)=\max_{y\in M} G(t,y)\quad \text{if and only if } x\in V. \]
    \end{enumerate}
    For example, $G$ can be a multiple of the function from Theorem~\ref{thm:precise1} by a function of $t$ with compact support in $(0,1)$.
    Note that the condition 2 implies that the time-average of $G$ at any $x\in V$ is the Hofer norm, $\int_0^1 G(t,x)\ dt = \|G\|_H$.
    \item $\tilde{G_k} :=  G \circ \psi \sharp \cdots \sharp G \circ \psi^k$. For $s\in [0,1]$, we have 
\[
\varphi_{s\tilde{G_k}}^t =\begin{cases}
     \psi^{-j} \circ  \varphi_{sG}^{\{kt\}} \circ \psi^j \circ \varphi_{s\tilde G_{j-1}},& t\in [\frac{j-1}{k}, \frac{j}{k}],\  j=1,\dots, k
     \end{cases}
\]     
where $\{kt\}$ is the fractional part of $kt$. Note also that the time-one map is given by $\varphi_{s\tilde G_{k}} = \psi^{-k}(\varphi_{sG} \psi)^{k}$.
    \item $F_k$ is some Hamiltonian generating $\psi^k$ that vanishes for $t$ near $0,1$.
    \item $g_s \colon M\rightarrow\R$ is defined by 
    \begin{equation}\label{eq:avg_function}
        g_s(x):=\int_0^1 G\big(t, \varphi_{sG}^{t-1}(x)\big) \ dt.
    \end{equation} 
    Note that $g_s\equiv G$ when $G$ is autonomous.
\end{itemize}
\end{setup}

\begin{lemma}
\label{lem:support}
    Assume Setup~\ref{set:setup_for_lemmas}. Every one-periodic orbit $x_s \colon \R /\Z \to M$ of $H_s := (s\tilde{G_k} )\sharp F_k$ satisfies
    \[
    \int_{x_s} \partial_s H_s  = k\cdot \int_M g_s \ d\mu^{\varphi_{sG}\psi}_{x_s(0)}.
    \]
\end{lemma}
\begin{proof}
We start by noticing that the Hamiltonian flow generated by the composition  $(s\tilde G_k)\sharp F_k$ is given by
\[
\varphi_{H_s}^t =\begin{cases}
     \psi^{-j} \circ  \varphi_{sG}^{\{2kt\}} \circ \psi^j \circ \varphi_{s\tilde G_{j-1}},& t\in [\frac{j-1}{2k}, \frac{j}{2k}],\  j=1,\dots, k\\
     \varphi_{F_k}^{2t-1} \circ \varphi_{s\tilde G_{k}} , & t\in [1/2,1]
\end{cases}
\]
where  
\[
\varphi_{s\tilde G_{j-1}} = \psi^{-(j-1)}(\varphi_{sG} \psi)^{j-1}.
\]
Note that $\varphi_{H_s} = (\varphi_{sG} \psi)^k$. Secondly, let us compute the $s$ derivative of $H_s$: 
\begin{align*}
    \partial_s H_s (t,x)=& \begin{cases}
    2\tilde G_k(2t,x), & t\in [0,1/2)\\
    0, & t\in [1/2,1]
\end{cases} \\
=& 
\begin{cases}
    2k\cdot G(2kt, \psi^j(x)), & t\in [\frac{j-1}{2k}, \frac{j}{2k}],\  j=1,\dots, k\\
    0, & t\in [1/2,1].
\end{cases}
\end{align*}
Overall we see that 
\begin{align*}
    \int_{x_s} \partial_s H_s &= \sum_{j=1}^k \int_{\frac{j-1}{2k}}^{\frac{j}{2k}} 2k\cdot G\Big(2kt,\ \psi^j \circ \varphi_{H_s}^t(x_s(0)) \Big)\ dt \\
    &= \sum_{j=1}^k \int_{\frac{j-1}{2k}}^{\frac{j}{2k}} 2k\cdot G\Big(2kt,\ \psi^j \circ \psi^{-j} \circ \varphi_{sG}^{\{ 2t\cdot k\}} \circ \psi^j \circ \varphi_{s\tilde G_{j-1}}  (x_s(0)) \Big)\ dt \\
    &= \sum_{j=1}^k \int_0^1 G\Big(t,\  \varphi_{sG}^{t} \circ \psi^j \circ \varphi_{s\tilde G_{j-1}}  (x_s(0)) \Big)\ dt \\
    &= \sum_{j=1}^k \int_0^1 G\Big(t,\ \varphi_{sG}^{t} \circ  \psi \circ (\varphi_{sG}\psi)^{j-1} (x_s(0))\Big)\ dt\\
    &= \sum_{j=1}^k \int_0^1 G\big(t,\ \varphi_{sG}^{t} \circ \varphi_{sG}^{-1} \circ (\varphi_{sG}\psi)^{j} (x_s(0))\big)\ dt.
\end{align*}
Recalling that  
\[
g_s(x)=\int_0^1  G\big(t,\varphi_{sG}^{t-1}(x) \big) \ dt,
\]
we see that 
\[\int_{x_s}\partial_sH_s = \sum_{j=1}^k  g_s \circ   (\varphi_{sG}\psi)^{j} (x_s(0)) =k\cdot \int_M g_s \ d\mu^{\varphi_{sG}\psi}_{x_s(0)}
\] 
where in the second equality we used the fact that $x_s(0)$ is a $k$-periodic point of $\varphi_{sG}\psi$. 
\end{proof}

\begin{lemma}
\label{lem:heavy}
Assume Setup~\ref{set:setup_for_lemmas}. We have:
$$
T^k_{\psi}(L, V)\cdot {\|G\|_H} - \gamma(F_k) \leq c(\tilde{G}_k \sharp F_k) - c(F_k).  
$$
\end{lemma}
\begin{proof}
    Our main task is to bound $c(\tilde G_k)$ from below, since 
    \[
        c(\tilde G_k) = c(\tilde G_k\sharp F_k\sharp \bar F_k) \leq c(\tilde G_k\sharp F_k)+c(\bar F_k)  = c(\tilde G_k\sharp F_k) - c(F_k) +\gamma(F_k).
    \]
    Observe that the required inequality will follow once we show that
    \[
    c(\tilde G_k)\geq T^k_\psi(L,V)\cdot \|G\|_H.
    \]
    This in turn relies on the heaviness  of $L$ which, by Lemma \ref{lem:heavy-time-dependent}, 
 yields the following inequality: 
    \begin{align*}
        c(\tilde G_k)  \geq \int_0^1\min_{x\in L} \tilde G_k(t,x)\ dt.
    \end{align*}  
    Recall that $\tilde G_k := G\circ \psi \sharp\cdots \sharp G\circ \psi^k$ and by definition of $\sharp$, 
    \[
    \tilde G_k(t,x) = k\cdot G(kt, \psi^j(x)) \quad \text{for} \quad t\in [\frac{j-1}{k},\frac{j}{k}].
    \]
    Due to our assumption, the restriction $G\vert_V$ only depends on $t$, and furthermore, for all $x\in V$, we have $\int_0^1G(t,x)\ dt=\|G\|_H$. It follows that
    \begin{align*}
        \int_0^1\min_{x\in L} \tilde G_k(t,x)\ dt  &= \sum_{j=1}^k \int_{\frac{j-1}{k}}^{\frac{j}{k}} k\cdot\min_{x\in L} G\big(kt, \psi^j(x)\big)\  dt\\
        &\geq \#\{1\leq j\leq k: \psi^j(L)\subset V\}\cdot \|G\|_H =T^k_\psi(L,V)\cdot \|G\|_H 
    \end{align*}
\end{proof}

\subsection{Proof of Theorem \ref{thm:precise1}}

Throughout this section we assume Setup \ref{set:setup_for_lemmas}. Observe that the autonomous Hamiltonian $G \colon M \to \R$ from the statement of Theorem \ref{thm:precise1} is not exactly as in the setup. To fix this we choose a smooth, monotone increasing and onto function $\lambda \colon [0,1] \to [0,1]$ with $\lambda'(t)\equiv0 $ near $t=0, 1$ and work with  $\lambda'G$ instead of $G$. Note that we have $\supp(G) = \supp(\lambda'G)$,  $\varphi^s_G= \varphi_{s\lambda'G}$ and $\max(G) = \|\lambda'G\|_H$. In what follows we denote by $\tilde{G}_{\lambda, k}$ the Hamiltonian  
$$
\tilde{G}_{\lambda, k} := \lambda' G \circ \psi \sharp \cdots \sharp \lambda' G \circ \psi^k.
$$

\begin{rem}
\label{rem:avg_function}
In this remark we observe that for a re-parametrized Hamiltonian  of the form above $\lambda' G$,  the associated average function $g_s \colon M \to \R$, see \eqref{eq:avg_function}, coincides with $G$. Hence, in this particular case,  Lemma \ref{lem:support} takes the following form: For every one-periodic orbit $x_s$ of $H_s= (s\tilde{G}_{\lambda, k})\sharp F_k$, we have
 \[
    \int_{x_s} \partial_s H_s  = k\cdot \int_M G \ d\mu^{\varphi^s_{G}\psi}_{x_s(0)}.
    \]
Note that we have also replaced $\varphi_{s\lambda'G}$ with $\varphi^s_G$. Let us show that this is indeed the case: 
\begin{align*}
    g_s(x):= \int_0^1 \lambda'(t)\cdot G(\varphi_{s\lambda'G}^{t-1}(x))\ dt &= \int_0^1 \lambda'(t)\cdot G(\varphi_{sG}^{\lambda(t)-1}(x))\ dt \\
    &=\int_0^1 \lambda'(t)\cdot G(x)\ dt =G(x)
\end{align*}
where we used $\varphi_{s\lambda'G}^{t-1}(x)=\varphi_{sG}^{\lambda(t)-1}$ and the fact that $G$ is invariant  under $\varphi_{sG}^{\lambda(t)-1}$.
\end{rem}

\noindent \underline{\textit{proof of part (i):}} Fix $k \in \N$ and let $F_k$ be a Hamiltonian generating $\psi^k$. For the family $H_s:=(s\tilde{G}_{\lambda,k} \sharp F_k)$, where $s\in [0,1]$, let $x_s \colon S^1 \to M$ be as in the first part of Lemma \ref{lem:derivative}. By Lemmas \ref{lem:derivative},  \ref{lem:support}, \ref{lem:heavy} and Remark \ref{rem:avg_function}, we have
\begin{align*}
    T^k_{\psi}(L, V) \|\lambda'G\|_H - \gamma(F_k) \leq&\  k\cdot \int_M G \ d\mu_{x_s(0)}^{\varphi^s_{G}\psi}.
\end{align*}
We replace $\|\lambda'G\|_H=\max(G)$ and re-organize:
\begin{equation}
\label{eq:pf_part1}
    \frac{1}{\max(G)}\int_M G \ d\mu_{x_s(0)}^{\varphi^s_{G}\psi} \geq \frac{T^k_{\psi}(L, V)}{k}- \frac{\gamma(F_k)}{k\cdot \max(G)}. 
\end{equation}
To conclude that there exists a $k$-periodic point $x$ of $\varphi^s_{G}\psi$ which satisfies \eqref{eq:thm_precise1_i}, we take the limit of \eqref{eq:pf_part1} over certain Hamiltonians generating $\psi^k$. Namely, let $F_k^i$ be a sequence of Hamiltonians, all generating $\psi^k$, and such that $\gamma(F_k^i) \to \gamma(\psi^k)$; and let $x_s^i$ be the respective one-periodic orbits satisfying \eqref{eq:pf_part1}. Observe that any limit point $x$ of the sequence $x_s^i(0)$ is a $k$-periodic point of $\varphi^s_{G}\psi$ and it satisfies \eqref{eq:thm_precise1_i}. This completes the proof of part (i). 
\medskip

\noindent \underline{\textit{proof of part (ii):}} Note that when $T_{\psi}^k(L, V)<\infty$, the recurrence frequency $\rho_{\psi}(L,V) = 0$. In this case, one can simply choose $S_k^{\alpha} = \emptyset$ for all $k\in \N$. In the rest of the argument, we work under the assumption that  $T_{\psi}^k(L, V) \to \infty $ as $k \to \infty$. For each $k \in \N$, choose a re-parametrization function $\lambda_k \colon [0,1] \to [0,1]$ which (additionally) satisfies
\begin{equation}
\label{eq:lambda}
    \max (\lambda'_k) \leq 1+1/k.
\end{equation}
Next, fix $0 <\alpha <1$, and set
\[
\eta_k := \alpha T_{\psi}^k(L, V) \| \lambda'_k G \|_H  - C_{\gamma}
\]
where $C_{\gamma} >0$ is the universal upper bound on the spectral pseudo-norm. Assume that $k \in \N$ is large enough so that $\eta_k> 0$. (Recall that $\| \lambda'_k G \|_H = \max(G)$.) For $k \in \N$ smaller than this threshold, similar to above, we set $S_k^{\alpha} = \emptyset$. 

For each $k \in \N$, choose a Hamiltonian $F_k$ generating $\psi^k$.  For the family $(s\tilde{G}_{\lambda_k, k} \sharp F_k)$, where $s\in [0,1]$, and $\eta_k >0$ chosen as above, let $S_k^{\alpha} := S_{\eta_k} \subset [0,1]$  be as in the second part of Lemma \ref{lem:derivative} and, for $s\in S_k^{\alpha}$, let $x_{k,s} \colon \R/ \Z \to M$ be the one-periodic orbit provided by the same lemma. Set 
\[
\Delta C_k := c(\tilde{G}_{\lambda_k, k}  \sharp F_k) - c(F_k).
\]
By Lemmas \ref{lem:derivative}, \ref{lem:support} and Remark \ref{rem:avg_function} we have
\[
k\cdot\int_M G \ d\mu_{x_s(0)}^{\varphi^s_G}\geq \Delta C_k -\eta_k 
\]
for all $s \in S_k^{\alpha}$. On the other hand, by Lemma \ref{lem:heavy}, we have
\begin{equation}
\label{eq:heavy_proof}
     \Delta C_k -\eta_k  \geq (1-\alpha)  T_{\psi}^k(L, V) \| \lambda'_k G \|_H.
\end{equation}
Recalling the definition of $\eta_k$, it follows that
\[
k\cdot\int_M G \ d\mu_{x_s(0)}^{\varphi^s_G}\geq   (1-\alpha) T_{\psi}^k(L, V) \| \lambda'_k G \|_H
\]
for all $s \in S_k^{\alpha}$. This is identical to \eqref{eq:thm_precise1_ii} since $ \| \lambda'_k G \|_H =\max (G)$. 

It remains to show that 
\begin{equation}
\label{eq:measure_proof}
     \limsup_{k \to \infty} m(S_k^{\alpha}) \geq \frac{\alpha \rho_{\psi}(L, V)}{1-(1-\alpha)\rho_{\psi}(L, U)}.
\end{equation}
By the second part of Lemma \ref{lem:derivative}, we have
\[
     m(S_k^{\alpha}) \geq \eta_k / (K_k-\Delta C_k + \eta_k)
\]
where 
\begin{equation}
\label{eq:bound_K}
  K_k := \max \partial_s (s\tilde{G}_{\lambda_k, k} \sharp F_k) = k \max (\lambda'_k)\max (G) \leq k (1+1/k) \max (G);  
\end{equation}
see \eqref{eq:lambda}. Using \eqref{eq:heavy_proof},  \eqref{eq:bound_K} and $\| \lambda'_k G \|_H = \max(G)$, we write
\begin{align*}
  \frac{\eta_k}{K_k-\Delta C_k + \eta_k} &\geq \frac{\alpha T_{\psi}^k(L, V) \max(G) - C_{\gamma}}{k (1+1/k) \max(G) - (1-\alpha) T_{\psi}^k(L, V) \max(G)} \\ 
  &= \frac{\alpha T_{\psi}^k(L, V)/k - C_{\gamma}/k\max(G)}{(1+1/k)- (1-\alpha) T_{\psi}^k(L, V)/k}.
\end{align*}
Taking the $\limsup$ of the right hand side yields \eqref{eq:measure_proof}. This completes the proof of Theorem \ref{thm:precise1}.

\bibliographystyle{plain}
\bibliography{refs}

\begin{thebibliography}{10}

\bibitem{ABC}
Flavio Abdenur, Christian Bonatti, and Sylvain Crovisier.
\newblock Nonuniform hyperbolicity for {$C^1$}-generic diffeomorphisms.
\newblock {\em Israel J. Math.}, 183:1--60, 2011.

\bibitem{asaoka2016c}
Masayuki Asaoka and Kei Irie.
\newblock A {$C^\infty$} closing lemma for {H}amiltonian diffeomorphisms of closed surfaces.
\newblock {\em Geometric and Functional Analysis}, 26(5):1245--1254, 2016.

\bibitem{CGG-spectral}
Erman \c{C}ineli, Viktor~L. Ginzburg, and Ba\c sak~Z. G\"urel.
\newblock On the generic behavior of the spectral norm.
\newblock {\em Pacific J. Math.}, 328(1):119--135, 2024.

\bibitem{cineli2022strong}
Erman \c{C}ineli and Sobhan Seyfaddini.
\newblock The strong closing lemma and {H}amiltonian pseudo-rotations.
\newblock {\em J. Mod. Dyn.}, 20:299--318, 2024.

\bibitem{CDPT}
Julian Chaidez, Ipsita Datta, Rohil Prasad, and Shira Tanny.
\newblock Contact homology and higher dimensional closing lemmas.
\newblock {\em J. Mod. Dyn.}, 20:67--153, 2024.

\bibitem{CT2023}
Julian Chaidez and Shira Tanny.
\newblock Elementary {SFT} spectral gaps and the strong closing property.
\newblock {\em arXiv:2312.17211}, 2023.

\bibitem{Conley-Zehnder}
C.~Conley and E.~Zehnder.
\newblock An index theory for periodic solutions of a {H}amiltonian system.
\newblock In {\em Geometric dynamics ({R}io de {J}aneiro, 1981)}, volume 1007 of {\em Lecture Notes in Math.}, pages 132--145. Springer, Berlin, 1983.

\bibitem{cristofaro2021smooth}
Dan Cristofaro-Gardiner, Rohil Prasad, and Boyu Zhang.
\newblock The smooth closing lemma for area-preserving surface diffeomorphisms.
\newblock {\em arXiv preprint arXiv:2110.02925}, 2021.

\bibitem{edtmair2021pfh}
Oliver Edtmair and Michael Hutchings.
\newblock {PFH} spectral invariants and {$C^\infty$} closing lemmas.
\newblock {\em arXiv preprint arXiv:2110.02463}, 2021.

\bibitem{entov2006quasi}
Michael Entov and Leonid Polterovich.
\newblock Quasi-states and symplectic intersections.
\newblock {\em Commentarii Mathematici Helvetici}, 81(1):75--99, 2006.

\bibitem{entov2009rigid}
Michael Entov and Leonid Polterovich.
\newblock Rigid subsets of symplectic manifolds.
\newblock {\em Compositio Mathematica}, 145(3):773--826, 2009.

\bibitem{FK}
Bassam Fayad and Anatole Katok.
\newblock Constructions in elliptic dynamics.
\newblock {\em Ergodic Theory Dynam. Systems}, 24(5):1477--1520, 2004.

\bibitem{fukaya2019spectral}
Kenji Fukaya, Yong-Geun Oh, Hiroshi Ohta, and Kaoru Ono.
\newblock Spectral invariants with bulk, quasi-morphisms and {L}agrangian {F}loer theory.
\newblock {\em Mem. Amer. Math. Soc.}, 260(1254):x+266, 2019.

\bibitem{GG}
Viktor~L. Ginzburg and Ba\c sak~Z. G\"urel.
\newblock Hamiltonian pseudo-rotations of projective spaces.
\newblock {\em Invent. Math.}, 214(3):1081--1130, 2018.

\bibitem{herman1991exemples}
M-R Herman.
\newblock Exemples de flots {H}amiltoniens dont aucune perturbation en topologie ${C}^{\infty}$ n'a d'orbites p{\'e}riodiques sur un ouvert de surfaces d'{\'e}nergies.
\newblock {\em CR Acad. Sci. Paris S{\'e}r. I Math.}, 312:989, 1991.

\bibitem{hofer2012symplectic}
Helmut Hofer and Eduard Zehnder.
\newblock {\em Symplectic invariants and Hamiltonian dynamics}.
\newblock Birkh{\"a}user, 2012.

\bibitem{HLeS}
Vincent Humili\`ere, Fr\'ed\'eric Le~Roux, and Sobhan Seyfaddini.
\newblock Towards a dynamical interpretation of {H}amiltonian spectral invariants on surfaces.
\newblock {\em Geom. Topol.}, 20(4):2253--2334, 2016.

\bibitem{i2015}
Kei Irie.
\newblock Dense existence of periodic {R}eeb orbits and {ECH} spectral invariants.
\newblock {\em Journal of Modern Dynamics}, 9(01):357, 2015.

\bibitem{Irie-equidist}
Kei Irie.
\newblock Equidistributed periodic orbits of {$C^\infty$}-generic three-dimensional {R}eeb flows.
\newblock {\em J. Symplectic Geom.}, 19(3):531--566, 2021.

\bibitem{JS}
Du\v{s}an Joksimovi\'c and Sobhan Seyfaddini.
\newblock A {H}\"older-type inequality for the {$C^0$} distance and {A}nosov-{K}atok pseudo-rotations.
\newblock {\em Int. Math. Res. Not. IMRN}, (8):6303--6324, 2024.

\bibitem{judd2019}
Jamie Judd and Konstanze Rietsch.
\newblock The tropical critical point and mirror symmetry.
\newblock {\em arXiv:1911.04463}, 2019.

\bibitem{kawamoto2024spectral}
Yusuke Kawamoto and Egor Shelukhin.
\newblock Spectral invariants over the integers.
\newblock {\em Advances in Mathematics}, 458:109976, 2024.

\bibitem{LCY}
Patrice Le~Calvez and Jean-Christophe Yoccoz.
\newblock Un th\'eor\`eme d'indice pour les hom\'eomorphismes du plan au voisinage d'un point fixe.
\newblock {\em Ann. of Math. (2)}, 146(2):241--293, 1997.

\bibitem{LeS}
Fr\'ed\'eric Le~Roux and Sobhan Seyfaddini.
\newblock The {A}nosov-{K}atok method and pseudo-rotations in symplectic dynamics.
\newblock In {\em Symplectic geometry---a {F}estschrift in honour of {C}laude {V}iterbo's 60th birthday}, pages 863--901. Birkh\"auser/Springer, Cham, [2022] \copyright 2022.

\bibitem{mane1982ergodic}
Ricardo Ma{\~n}{\'e}.
\newblock An ergodic closing lemma.
\newblock {\em Annals of Mathematics}, 116(3):503--540, 1982.

\bibitem{McDuff-uniruled}
Dusa McDuff.
\newblock Hamiltonian {$S^1$}-manifolds are uniruled.
\newblock {\em Duke Math. J.}, 146(3):449--507, 2009.

\bibitem{McDuff-Salamon-Jcurves}
Dusa McDuff and Dietmar Salamon.
\newblock {\em J-holomorphic curves and symplectic topology}, volume~52.
\newblock American Mathematical Soc., 2012.

\bibitem{oh}
Yong-Geun Oh.
\newblock Construction of spectral invariants of {H}amiltonian paths on closed symplectic manifolds.
\newblock In {\em The breadth of symplectic and {P}oisson geometry}, volume 232 of {\em Progr. Math.}, pages 525--570. Birkh\"{a}user Boston, Boston, MA, 2005.

\bibitem{Oh09spectral}
Yong-Geun Oh.
\newblock Floer mini-max theory, the {C}erf diagram, and the spectral invariants.
\newblock {\em J. Korean Math. Soc.}, 46(2):363--447, 2009.

\bibitem{OT-toric}
Yaron Ostrover and Ilya Tyomkin.
\newblock On the quantum homology algebra of toric {F}ano manifolds.
\newblock {\em Selecta Math. (N.S.)}, 15(1):121--149, 2009.

\bibitem{polterovich2014function}
Leonid Polterovich and Daniel Rosen.
\newblock {\em Function theory on symplectic manifolds}, volume~34.
\newblock American Mathematical Soc., 2014.

\bibitem{prasad2022}
Rohil Prasad.
\newblock Generic equidistribution of periodic orbits for area-preserving surface maps.
\newblock {\em arXiv:2112.14601}, 2022.

\bibitem{Pugh1}
Charles~C. Pugh.
\newblock The closing lemma.
\newblock {\em Amer. J. Math.}, 89:956--1009, 1967.

\bibitem{Pugh2}
Charles~C. Pugh.
\newblock An improved closing lemma and a general density theorem.
\newblock {\em Amer. J. Math.}, 89:1010--1021, 1967.

\bibitem{Pugh-Robinson}
Charles~C. Pugh and Clark Robinson.
\newblock The {$C^{1}$} closing lemma, including {H}amiltonians.
\newblock {\em Ergodic Theory Dynam. Systems}, 3(2):261--313, 1983.

\bibitem{schwarz}
Matthias Schwarz.
\newblock On the action spectrum for closed symplectically aspherical manifolds.
\newblock {\em Pacific J. Math.}, 193(2):419--461, 2000.

\bibitem{usher2008spectral}
Michael Usher.
\newblock Spectral numbers in {F}loer theories.
\newblock {\em Compos. Math.}, 144(6):1581--1592, 2008.

\bibitem{usher2011deformed}
Michael Usher.
\newblock Deformed hamiltonian {F}loer theory, capacity estimates and calabi quasimorphisms.
\newblock {\em Geometry \& Topology}, 15(3):1313--1417, 2011.

\bibitem{viterbo92}
Claude Viterbo.
\newblock Symplectic topology as the geometry of generating functions.
\newblock {\em Math. Ann.}, 292(4):685--710, 1992.

\bibitem{Xue}
Jinxin Xue.
\newblock Strong closing lemma and {KAM} normal form.
\newblock {\em arXiv:2207.06208}, 2022.

\bibitem{Zehnder-torus}
E.~Zehnder.
\newblock Remarks on periodic solutions on hypersurfaces.
\newblock In {\em Periodic solutions of {H}amiltonian systems and related topics ({I}l {C}iocco, 1986)}, volume 209 of {\em NATO Adv. Sci. Inst. Ser. C: Math. Phys. Sci.}, pages 267--279. Reidel, Dordrecht, 1987.

\end{thebibliography}

\medskip
 \noindent Erman \c C\. inel\. i\\
\noindent D-MATH, ETH Zürich, Rämistrasse 101, 8092 Zürich, Switzerland.\\
 {\it e-mail:} erman.cineli@math.ethz.ch

\medskip
 \noindent Sobhan Seyfaddini\\
\noindent D-MATH, ETH Zürich, Rämistrasse 101, 8092 Zürich, Switzerland.\\
 {\it e-mail:} sobhan.seyfaddini@math.ethz.ch

 \medskip
 \noindent Shira Tanny\\
\noindent School of Mathematics, Weizmann Institute of Science, Rehovot, Israel.\\
 {\it e-mail:} shira.tanny@weizmann.ac.il
 
\end{document}